\documentclass[12pt, a4paper, leqno]{amsart}
\usepackage[utf8]{inputenc}
\usepackage{amsmath}
\usepackage{amsfonts}
\usepackage{amssymb}
\usepackage{times}
\usepackage[T1]{fontenc} 
\usepackage{url} 
\usepackage{color,esint}
\usepackage[dvipsnames]{xcolor}
\usepackage[colorlinks, allcolors=RedViolet,pdfstartview=,pdfpagemode=UseNone]{hyperref} 
\usepackage{graphicx} 
\usepackage{enumitem}
\usepackage{tabularx}
\usepackage{mathtools}
\usepackage{pgf,tikz}
\usepackage{mathrsfs}
\usetikzlibrary{arrows}
\usepackage{array,multirow}
\usepackage{booktabs}
\newcommand{\bigcell}[2]{\begin{tabular}{@{}#1@{}}#2\end{tabular}}

\let\oldmarginpar\marginpar
\renewcommand\marginpar[1]{\-\oldmarginpar[\raggedleft\footnotesize #1]%
{\raggedright\footnotesize #1}}

\usepackage{amsthm}
\theoremstyle{plain}
\newtheorem{thm}[equation]{Theorem}
\newtheorem{lem}[equation]{Lemma}
\newtheorem{prop}[equation]{Proposition}
\newtheorem{cor}[equation]{Corollary}

\theoremstyle{definition}
\newtheorem{defn}[equation]{Definition}

\newtheorem{eg}[equation]{Example}

\theoremstyle{remark}
\newtheorem{rem}[equation]{Remark}

\numberwithin{equation}{section}

\newcommand{\R}{\mathbb{R}}

\newcommand{\Rn}{{\mathbb{R}^n}}

\renewcommand{\phi}{\varphi}
\def\le{\leqslant}
\def\leq{\leqslant}
\def\ge{\geqslant}
\def\geq{\geqslant}
\def\phi{\varphi}
\def\rho{\varrho}
\def\vartheta{\theta}

\def\ve{\varepsilon}
\newcommand{\Phiw}{\Phi_{\text{\rm w}}}
\newcommand{\Phic}{\Phi_{\text{\rm c}}}
\newcommand{\Phis}{\Phi_{\text{\rm s}}}

\def\dist{\qopname\relax o{dist}}

\def\loc{{\rm loc}}

\newcommand{\inc}[1]{\hyperref[def:aInc]{{\normalfont(Inc){\ensuremath{_{#1}}}}}}
\newcommand{\dec}[1]{\hyperref[def:aDec]{{\normalfont(Dec){\ensuremath{_{#1}}}}}}
\newcommand{\ainc}[1]{\hyperref[def:aInc]{{\normalfont(aInc){\ensuremath{_{#1}}}}}}
\newcommand{\ainci}[1]{\hyperref[def:aInc]{{\normalfont(aInc){\ensuremath{_{#1}^\infty}}}}}
\newcommand{\adec}[1]{\hyperref[def:aDec]{{\normalfont(aDec){\ensuremath{_{#1}}}}}}
\newcommand{\adeci}[1]{\hyperref[def:aDec]{{\normalfont(aDec){\ensuremath{_{#1}^\infty}}}}}
\newcommand{\azero}{\hyperref[def:a0]{{\normalfont(A0)}}}
\newcommand{\aone}{\hyperref[def:a1]{{\normalfont(A1)}}}
\newcommand{\aonep}{\hyperref[def:a1p]{{\normalfont(A1')}}}
\newcommand{\aones}[1]{\hyperref[def:a1s]{{\normalfont(A1-{\ensuremath{{#1}})}}}}

\newcommand{\atwo}{\hyperref[def:a2]{{\normalfont(A2)}}}

\date{\today}

\begin{document}

\title
[Sharp growth conditions for boundedness of maximal function]
{Sharp growth conditions for boundedness of maximal function in generalized Orlicz spaces}
\author[Harjulehto and Karppinen]{Petteri Harjulehto and Arttu Karppinen}


\begin{abstract}
We study sharp growth conditions for the  boundedness of the Hardy--Littlewood maximal function in the generalized Orlicz spaces.
We assume that the generalized Orlicz function $\phi(x, t)$ satisfies the standard continuity properties (A0), (A1) and (A2).
We show that if the Hardy--Littlewood maximal function is bounded from the generalized Orlicz space to itself then $\phi(x,t)/ t^p$ is almost increasing for large $t$ for some $p>1$. Moreover we show that 
the Hardy--Littlewood maximal function is bounded from the generalized Orlicz space $L^\phi(\Rn)$ to itself if and only if $\phi$ is weakly equivalent to a generalized Orlicz function $\psi$ satisfying (A0), (A1) and (A2) for which $\psi(x,t)/ t^p$ is almost increasing for all $t>0$ and some $p>1$.
\end{abstract}

\maketitle


\section{Introduction}

The celebrated  Hardy--Littlewood maximal function 
\begin{align*}
Mf(x) = \sup_{B \ni x} \dfrac{1}{|B|} \int_{B} |f(x)| \, dx
\end{align*}
is known to be bounded in the classical Lebesgue spaces $L^p(\Rn)$ if and only if $p>1$. This was first proved by Hardy and Littlewood \cite{HarL30} in one dimensional case, and by Wiener  \cite{Wie39} in the case $n \ge 2$.

Gallardo \cite{Gal88} proved in the classical Orlicz spaces that the Hardy--Littlewood maximal 
function is bounded   if and only if  the complementary function satisfies $\Delta_2$-condition.  This characterization is for N-functions.
The complementary function $\phi^*$ satisfies $\Delta_2$-condition if and only if there exists $p>1$ such 
that $\frac{\phi(t)}{t^p}$ is almost increasing for all $t>0$. We call the latter condition \ainc{},  for
 the precise definition see Section~2. For some recent studies in this topic, see for example \cite{Kit97, Mus19}. 

In the variable exponent space $L^{p(\cdot)}(\Rn)$ the Hardy--Littlewood maximal function is bounded 
provided that $\frac1p$ is globally Hölder continuous and $p^-:= \inf p(x) >1$. The right modulus of continuity 
for $p$ was first observed by Diening \cite{Die04}, and for later development see example  
\cite{AlmHHL15, CruFN03, HarHMS11, Ler05, Nek04, Nek19, PicR01}. The condition $p^- >1$, not $p(x) >1$ for all $x$, is also necessary 
for the boundedness of the maximal function \cite[Theorem~6.3]{DieHHMS09}. Note that if $p^->1$ then $ \frac{t^{p(x)}}{t^{p^-}}$ is increasing for every $x$ and hence \ainc{} holds.

Generalized Orlicz spaces, also known as Musielak--Orlicz spaces, have gained steady increase of interest over the last years. 
Many problems of harmonic analysis and regularity theory have been studied in this setting.
One of the main motivations for these studies is to create a unified theory of previously mentioned and other function spaces. 
Many standard results, such as necessity of superlinear growth rate for boundedness of the maximal function or H\"older continuity of a solution to an elliptic 
partial differential equation, have had noticeably different approaches between Orlicz and variable exponent spaces, see for example \cite{Alk97,DieHHR11,Lie91}. 

In the generalized Orlicz spaces the integrability is given by a function $\phi(x, f(x))$. For the 
definitions see Section~2. Hästö \cite{Has15} showed that 
the Hardy--Littlewood maximal function is bounded from the generalized Orlicz spaces to itself provided that $\phi$ 
satisfies the standard assumptions \azero{}, \aone{} and \atwo{}, and also \ainc{}. For the  former result and further developments see for example 
\cite{Die05, HarH17, MaeMOS13, MaeMOS13b, OhnS14, OhnS18}.
In this article we study here the necessity of \ainc{}.

We aim to give a complete picture of the relation between almost increasingness of $\frac{\phi(x,t)}{t^p}$ (the \ainc{p} condition) 
and boundedness of the Hardy--Littlewood maximal function. 
Contrary to known special cases, in generalized Orlicz spaces the boundedness of the maximal function does not imply that $\phi$ satisfies \ainc{p} for all $t$ 
and some $p>1$, see Theorem \ref{thm:bounded}. 
A concrete $\Phi$-function demonstrating this phenomenon is given in Example \ref{eg:phi-without-ainc}. The measure of the set where \ainc{p} fails plays a significant role as shown in Example \ref{eg:infinite-G}, where the maximal function is not bounded with similar assumptions on the function $\phi$. 

In Theorem \ref{thm:ainc-infty} we show that boundedness  of maximal function implies that $\phi$ satisfies \ainc{p} for 
any $t\geq t_0 >0$ and some $p>1$. 
To achieve a natural growth rate also for small $t$ we need to modify the function $\phi$ somehow.
As the first option we show in Proposition \ref{prop:p_infty} that an asymptotic (not generalized) Orlicz function from 
the assumption \atwo{} satisfies \ainc{p} 
for every $t$. 
The second modification option is to have a weakly equivalent generalized Orlicz function $\psi$. 
In Theorem \ref{thm:weak} we show that boundedness of the maximal function is equivalent to existence of a weakly 
equivalent generalized Orlicz function 
satisfying \ainc{p} for all $t\geq 0$ and some $p>1$.

We have summarized our results in the following table. All of the assumptions include that $\phi$ satisfies \azero{}, \aone{} and \atwo{}. 

{\small
\begin{centering}
\begin{table}[h]
\begin{tabular}{|l|l|l|}
\hline
Assumptions   & Outcome         & Result                   \\ \hline

 \bigcell{l}{$\phi$ satisfies  \ainc{} for large $t$ in $\Omega$ and \\ \ainc{} fails for small $t$ in $G$, $|G|<\infty$}& $\Rightarrow M$ is bounded  in $L^\phi(\Omega)$      & Theorem \ref{thm:bounded} \\ \hline  \hline
                                                                                                                                      
\multirow{2}{*}{\begin{tabular}[c]{@{}l@{}}$\phi$ satisfies  \ainc{} for large $t$ in $\Rn$ and \\ \ainc{} fails for small $t$ in $\Rn$\end{tabular}}                      & $\Rightarrow M$ can be unbounded  in $L^\phi(\Rn)$                   &                         Example \ref{eg:infinite-G}  \\ \cline{2-3} 
                                                                                                                                      & $ \Rightarrow M$ can be bounded  in $L^\phi(\Rn)$                 &                         Example \ref{eg:M-bdd-infinite-measure} \\ \hline  \hline
                                                                                                                                     \multirow{4}{*}{\begin{tabular}[c]{@{}l@{}}$M$ is bounded in $L^\phi(\Rn)$\end{tabular}} & $\Rightarrow \phi$ satisfies \ainc{} for large $t$ & Theorem \ref{thm:ainc-infty}            \\ \cline{2-3} 
                                                                                                                                      & $\Rightarrow $ Modified $\phi_\infty$  satisfies \ainc{} for all $t$      & Proposition \ref{prop:p_infty}         \\ \cline{2-3} 
                                                                                                                                      & $\Leftrightarrow \exists \ \psi$ satisfying \azero{}, \aone{}, \atwo{} & \multirow{2}{*}{Theorem \ref{thm:weak}} \\
                                                                                                                                      &                 and \ainc{} for all $t$   such that $\phi \sim \psi$          &                          \\ \hline                                                                                                                                  
\end{tabular}
\end{table}
\end{centering}
}

Let us describe the essence of our proofs. We start by showing \ainc{p} of $\phi$ for large $t$ and \ainc{p} of $\phi_\infty$ for small $t$. 
We also show, that the range of validity for \ainc{p} can be always enlarged to any interval $(t_0, t_1)$, where  $0<t_0<t_1<\infty$.
Therefore we can match the ranges of $t$ between $\phi$ and $\phi_\infty$ and glue them together to obtain a new generalized function satisfying \ainc{p} for all $t \geq 0$.
The counter examples in Sections 3 and 4 also employ glueing of function, but this time with linear or quadratic growth. This time the functions are combined at a point depending on the value of the function and $x$. 
This way the glueing is done at different level sets at different points in space and therefore we can change the growth rate of $\phi$ according to asymptotic behaviour of the maximal function. 

This flexible interplay between $x$ and $t$ is not possible in many of the known non-autonomous special cases such as with variable exponent or double phase growth rates. 
Therefore these special cases do not rule out for necessary assumption of boundedness of the maximal function to hold for small $t$,  
see \cite[Theorem~6.3]{DieHHMS09} and Proposition \ref{cor:dp-case}.

\section{Generalized $\Phi$-functions}

\textit{L-almost increasing} means that a function satisfies $f(s) \le L f(t)$ for all $s<t$ and some constant $L\ge 1$.
 \textit{L-almost decreasing} is defined analogously.
If there exists a constant $C$ such that $f(x) \leq C g(x)$ for almost every $x$, then we write $f \lesssim g$. If $f\lesssim g\lesssim f$, then we write $f \approx g$. In this article $\Omega \subset \Rn$ is an open set. For any measurable set $A$ we denote $\chi_A$ as its characteristic function and $|A|$ as its Lebesgue measure. If $X$ and $Y$ are normed spaces, the norm in $X\cap Y$ is defined as $\|\cdot\|_{X \cap Y} = \max\{\|\cdot\|_X, \|\cdot\|_Y\}$.

\begin{defn}
We say that $\phi: \Omega\times [0, \infty) \to [0, \infty]$ is a 
\textit{weak (generalized) $\Phi$-function}, and write $\phi \in \Phiw(\Omega)$, if 
the following conditions hold:
\begin{itemize}
\item 
For every measurable function $f:\Omega\to \R$ the function $x \mapsto \phi(x, f(x))$ is measurable 
and for every $x \in \Omega$ the function $t \mapsto \phi(x, t)$ is non-decreasing.
\item 
$\displaystyle \phi(x, 0) = \lim_{t \to 0^+} \phi(x,t) =0$  and $\displaystyle \lim_{t \to \infty}\phi(x,t)=\infty$ for every $x\in \Omega$.
\item 
The function $t \mapsto \frac{\phi(x, t)}t$ is $L$-almost increasing on $(0,\infty)$ with $L$ independent of $x$.
\end{itemize}
If $\phi\in\Phiw(\Omega)$ and  additionally $t \mapsto \phi(x, t)$  is convex and left-continuous for almost every $x$, then $\phi$ is a 
\textit{convex $\Phi$-function}, and we write $\phi \in \Phic(\Omega)$. 
If $\phi\in\Phiw(\Omega)$ and  additionally $t \mapsto \phi(x, t)$  is  convex and  continuous in the topology of $[0, \infty]$ for almost every $x$, then $\phi$ is a 
\textit{strong $\Phi$-function}, and we write $\phi \in \Phis(\Omega)$. 

If $\phi$ does not depend on $x$, then we omit the set  and write $\phi \in \Phiw$, $\phi \in \Phic$
or $\phi\in \Phis$.
\end{defn}

A function $\phi \in \Phi_c$ is called N-function if $\phi(t) \in (0, \infty)$ for all $t>0$, $\lim_{t\to 0^+}\frac{\phi(t)}{t}=0$ and
$\lim_{t\to \infty}\frac{\phi(t)}{t}=\infty$. N-function is always continuous, since it is finite and convex, and thus it is a  strong $\Phi$-function.

Two functions $\phi$ and $\psi$ are \textit{equivalent}, 
$\phi\simeq\psi$, if there exists $L\ge 1$ such that 
$\psi(x,\frac tL)\le \phi(x, t)\le \psi(x, Lt)$ for every $x \in \Omega$ and every $t>0$.
Equivalent $\Phi$-functions give rise to the same space with 
comparable norms.  

Two functions $\phi$ and $\psi$ are \textit{weakly equivalent}, 
$\phi\sim\psi$, if there exists $L\ge 1$  and $h \in L^1(\Omega)$ such that 
$\psi(x,t)\le \phi(x, Lt) + h(x)$ and $\phi(x, t)\le \psi(x, Lt) + h(x)$ for all $t \ge 0$ and almost all $x \in \Omega$.
Weakly equivalent $\Phi$-functions give rise to the same space with 
comparable norms.

By $\phi^{-1}(x,t)$ we mean a generalized inverse defined by
\begin{align*}
\phi^{-1}(x,t) := \inf\{\tau \ge 0 : \phi(x,\tau) \geq t\}.
\end{align*}

For $\phi \in \Phiw(\Omega)$ we define a conjugate $\phi$-function $\phi^{\ast} \in \Phiw(\Omega)$ by
\begin{align*}
\phi^\ast(x,t) := \sup_{s>0} \big(st - \phi(x,s) \big).
\end{align*}
It is also noteworthy that $\phi^{\ast} \in \Phi_c(\Omega)$ always.

We collect some results how the generalized inverse behaves. For the proofs see
Lemma 2.3.3, Lemma 2.3.9 and Theorem 2.4.8 in \cite{HarH_book}.

\begin{lem}\label{lem:composition}\label{lem:conjugate-inverse}
\begin{itemize}
\item[(a)]  $\phi^{-1}(x, \phi(x,t)) \leq t$.
\item[(b)] $\phi^{-1}(x,\phi(x,t)) \approx t$ when $\phi(x,t) \in (0, \infty)$.
\item[(c)]  $\phi(x, \phi^{-1}(x, t)) =t$ when $\phi \in \Phis(\Omega)$,
\item[(d)]  $\phi^{-1}(x,t) \left (\phi^{\ast}\right )^{-1}(x,t) \approx t$.
\end{itemize} 
\end{lem}

\begin{defn}\label{def:aInc}
Let $A \subset \Omega$ and $T \subset [0, \infty)$
We say that $\phi \in \Phiw(\Omega)$ satisfies \ainc{p} on $A \times T$ if there exists a constant 
$a\ge 1$ such that 
\[
\frac{\phi(x, t)}{t^p} \le a \frac{\phi(x, s)}{s^p}
\]
for all $t, s \in T$, $t< s$, and almost all $x \in A$. If no set is mentioned, then \ainc{p} is assumed to hold for all $s>t>0$ and almost all $x \in \Omega$.
We say that \ainc{} holds if \ainc{p} holds for some $p>1$.
We say that \ainci{} holds  if there exists $t_0$ such that $\phi$ satisfies \ainc{} on $\Omega \times [t_0, \infty)$.
\end{defn}

\begin{lem}\label{lem:ainc-equivalent}
Let $T \subset (0, \infty)$ be an interval.
Let $\phi:\Omega\times [0,\infty)\to [0,\infty)$ and $\psi:\Omega\times [0,\infty)\to [0,\infty)$ be increasing with $\phi \simeq \psi$.
If $\phi$ satisfies \ainc{p} on $A \times T$, then $\psi$ satisfies \ainc{p} on $A \times T$.
\end{lem}

\begin{proof}
Let $x \in A$.
 Let $s, t \in T$ with $s<t$ and assume first that $L^2s <t$, where $L\ge 1$ is the constant from the equivalence.
 Note that then we have $ s \le Ls \le t/L \le t$, and hence $Ls, t/L \in T$.
By \ainc{p} of $\phi$ with a constant $a$, we obtain
\[
\frac{\psi(x,s)}{s^p} \le  L^p\frac{\phi(x,Ls)}{(Ls)^p} \le a L^p \frac{\phi(x,t/L)}{(t/L)^p}
\le a L^{2p} \frac{\psi(x,t)}{t^p}.
\] 
Assume then that $t \in (s, L^2 s]\cap T$.
Using that $\psi$ is increasing, we find that 
\[
\frac{\psi(x,s)}{s^p} \le \frac{\psi(x,t)}{s^p}
= \frac{t^p}{s^p}\frac{\psi(x,t)}{t^p} \le L^{2p} \frac{\psi(x,t)}{t^p}.
\] 
We have shown that $\psi$ satisfies \ainc{p} with a constant 
$a L^{2p}$ and for the same range of $t$ than $\phi$.
\end{proof}

We say that $\phi:\Omega\times [0,\infty)\to [0,\infty)$ satisfies 
\begin{itemize}
\item[(aDec)$_q$] \label{def:aDec}
if
$t \mapsto \frac{\phi(x,t)}{t^{q}}$ is $L_q$-almost 
decreasing in $(0,\infty)$ for some $L_q\ge 1$ and a.e.\ $x\in\Omega$.
\end{itemize}

Conditions \ainc{} and \adec{} correspond to the $\nabla_2$ and $\Delta_2$ conditions respectively from the classical 
Orlicz space theory. 
The result for the next lemma shows this correspondence for $\nabla_2$ and \ainc{} with the additional restriction ''for 
$t>t_0$''. Traditionally $\nabla_2$ is formulated as $\phi^\ast$ satisfying $\Delta_2$ condition, but here it is 
formulated without any mentions of the conjugate function. This definition has been used in the literature for example in 
\cite{BaaBO20}.

\begin{lem}\label{lem:Sun-Sig}
Let $\phi \in \Phiw(\Omega)$. Then
$\phi$ satisfies \ainci{} if and only if 
there exist $c> 1$ and  $t_0>0$ such that $2c\phi(x, t)\le \phi(x,ct)$ 
for all $t\ge t_0$ and almost all $x \in \Omega$.
\end{lem}

\begin{proof}
Assume first that $\phi$ satisfies \ainci{p}, $p>1$. Then for $t_0 \le t<s$ we have
\[
\frac{\phi(x, t)}{t^p} \le a \frac{\phi(x, s)}{s^p}.
\]
Let us choose $c:= (2a)^{1/(p-1)}$ and $s:=ct$. Then a straight calculation gives 
$ 2 c \phi(x, t) \le \phi(x, c t)$ for $t\ge t_0$.

Assume then that $2c\phi(x, t)\le \phi(x,ct)$  for all $t\ge t_0$.  Let  $s>t\ge t_0$. Choose an integer $k\ge 1$ such that $c^{k-1}t < s \le c^{k} t$. Then
\[
\phi(x, t) \le \frac{1}{2c} \phi(x, ct) \le \frac{1}{(2c)^2} \phi(x, c^2t)
\le \ldots \le \frac{1}{(2c)^{k-1}} \phi(x, c^{k-1}t) \le \frac{1}{(2c)^{k-1}} \phi(x, s).
\]
Let $p>1$. Then the previous inequality with  $s \le c^{k} t$ yields
\[
\frac{\phi(x, t)}{(c^{k}t)^p} \le  \frac{1}{(2c)^{k-1}} \frac{\phi(x, s)}{(c^{k}t)^p}
\le \frac{1}{(2c)^{k-1}} \frac{\phi(x, s)}{s^p}
\]
and furthermore
\[
\frac{\phi(x, t)}{t^p} \le \frac{c^{pk}}{(2c)^{k-1}} \frac{\phi(x, s)}{s^p}
\le c^p \Big(\frac{c^{p}}{2c}\Big)^{k-1} \frac{\phi(x, s)}{s^p}.
\]
Then we choose $p$ so that  $\frac{c^{p}}{2c}=1$ i.e.  $p= \frac{\log(2)}{\log(c)}+1>1$.
\end{proof}

The next lemma shows that the set $T$ in \ainci{} can be always enlarged.

\begin{lem}\label{lem:ainc-for-smaller}
Let $A \subset \Omega$,  $p \ge 1$ and $0<t_1 < t_2$.\\ 
(a) If $\phi\in \Phiw(\Omega)$ satisfies \ainc{p} on  $A \times[t_2, \infty)$ with a constant $a$, then
$\phi$ satisfies \ainc{p} on  $A \times [t_1, \infty)$ with a constant $a^2 \big(\frac{t_2}{t_1} \big)^{p-1}$.\\
(b) If $\phi\in \Phiw(\Omega)$ satisfies \ainc{p} on  $A \times (0, t_1]$ with a constant $a$, then
$\phi$ satisfies \ainc{p} on  $A \times (0, t_2]$ with a constant $a^2 \big(\frac{t_2}{t_1} \big)^{p-1}$.
\end{lem}

\begin{proof}
(a) Let $x \in A$. Assume first that $t_1 \le t \le s \le t_2$. Then by \ainc{1}
\[
\frac{\phi(x, t)}{t^p} \le a \frac1{t^{p-1}}\frac{\phi(x, s)}{s} = a\frac{s^{p-1}}{t^{p-1}}\frac{\phi(x, s)}{s^p} \le a\frac{t_2^{p-1}}{t_1^{p-1}}\frac{\phi(x, s)}{s^p}.
\]
If $t_1 \le t \le  t_2 \le s$, then we obtain by the previous case and the assumption that
\[
\frac{\phi(x, t)}{t^p} \le a \frac{t_2^{p-1}}{t_1^{p-1}}\frac{\phi(x, t_2)}{t_2^p} \le  a^2\frac{t_2^{p-1}}{t_1^{p-1}}\frac{\phi(x, s)}{s^p}. 
\]

(b) The proof is similar than in case (a).
\end{proof}

We define several conditions. 
\begin{defn}
We say that $\phi:\Omega\times [0,\infty)\to [0,\infty)$ satisfies 
\begin{itemize}
\item[(A0)]\label{def:a0}
if there exists $\beta \in(0, 1]$ such that $ \beta \le \phi^{-1}(x,
1) \le \frac1{\beta}$ for almost every $x \in \Omega$;

\item[(A1)] \label{def:a1}
if there exists $\beta\in (0,1)$ such that
\[
\beta \phi^{-1}(x, t) \le \phi^{-1} (y, t)
\]
for every $t\in [1,\frac 1{|B|}]$, almost every $x,y\in B \cap \Omega$ and 
every ball $B$ with $|B|\le 1$;

\item[(A1')] \label{def:a1p}
if there exists $\beta\in (0,1)$ such that
\[
\phi(x, \beta t) \le \phi (y, t)
\]
for every $\phi(y, t)\in [1,\frac 1{|B|}]$, almost every $x,y\in B \cap \Omega$ and 
every ball $B$ with $|B|\le 1$;

\item[(A2)]\label{def:a2}
if there exists $\phi_\infty \in \Phiw$, $h \in L^1(\Omega)\cap L^\infty(\Omega)$, $\beta \in(0, 1]$ and $s>0$ such that
\[
\phi(x, \beta t) \le \phi_\infty(t) + h(x) \quad \text{ and } \quad \phi_\infty(\beta t) \le \phi(x, t) + h(x) 
\] 
for almost every $x \in \Omega$ when $\phi_\infty(t) \in [0, s]$ and $\phi(x, t) \in [0,s]$, respectively.
\end{itemize} 
\end{defn}

\begin{rem}\label{rem:assumptions}
(a) The conditions \azero{} - \atwo{} are invariant under equivalence ($\simeq$). However, this is not true for weak equivalence ($\sim$) in general.

(b) \azero{} holds if and only if there exists $\beta \in (0, 1]$ such that $\phi(x, \beta ) \le 1$ and $\phi(x, 1/\beta) \ge 1$ for almost every $x \in \Omega$.  

(c) \aone{} implies \aonep{}. \azero{} and \aonep{} imply \aone{}. See Proposition 4.1.5 and Corollary 4.1.6 in \cite{HarH_book}. 

(d) If conditions \azero{} - \atwo{} hold for $\phi$, then they hold also for $\phi^\ast$ (see \cite{HarH_book}).

(e) If \atwo{} holds for some $s>0$, then it holds for $s=1$, see \cite[Lemma 4.2.9]{HarH_book}
\end{rem}

Ranges for the variable $t$ are crucial in the assumptions \aone{} and \atwo{}. For example if we assume \aone{} to hold for all $t<1$ also, 
it  is not equivalent to  
the sharp regularity conditions of the special cases such as $\log$-H\"older continuity for variable exponent case. In Table~\ref{tab:A-ehdot} 
we have collected  conditions that imply \azero{}, \aone{}, \atwo{} and \ainc{} in the special cases. For the proof see Chapter 7 in \cite{HarH_book} 
and in the case of variable exponent double phase \cite{CreGHW}.

{\small
\begin{table}[ht!]
\centerline{\setlength{\tabcolsep}{3pt}
\renewcommand{\arraystretch}{1.2}
\begin{tabular}{l|cccc}
$\phi(x,t)$ &  \azero{} & \aone{} & \atwo{} & \ainc{} \\
\hline
$\phi(t)$ & \text{true} & \text{true} & \text{true} & \text{same}\\
$t^{p(x)}a(x)$ & $a\approx 1$ & $\frac1p\in C^{\log }$ & Nekvinda 
& $p^->1$ \\
$t^p + a(x) t^q$ & $a\in L^\infty$ & $a\in C^{0, \frac np (q-p)}$ &$\text{true}$ 
& $p>1$ \\
$t^{p(x)} + a(x) t^{q(x)}$ & $a\in L^\infty$ &  
$\begin{cases} a \in C^{0, \alpha},\\ q \in C^{0, \frac{\alpha}{q^-}}, p\in C^{\log } \\ \frac{q(x)}{p(x)}\le 1 + \frac{\alpha}{n} \end{cases}$
&$\text{true}$ & $p^->1$ 
\end{tabular}}
\caption{Conditions in special cases.}\label{tab:A-ehdot}
\end{table}
}

The generalized Orlicz space $L^{\phi}(\Omega)$ consists of measurable functions $f$ satisfying
\begin{align*}
\int_{\Omega} \phi(x, \lambda f(x)) \, dx < \infty
\end{align*}
for some $\lambda >0$. It is a quasi Banach space when equipped with a (quasi)norm
\begin{align*}
\|f\|_{L^{\phi}(\Omega)} := \inf \left \{ \lambda >0 : \int_{\Omega} \phi\left (x, \frac{f(x)}{\lambda}\right ) \, dx \leq 1 \right \}.
\end{align*}
If the set $\Omega$ is understood, we abbreviate $\|f\|_{L^\phi(\Omega)}$ as $\|f\|_{\phi}$.

\begin{lem}[Lemma 3.1.3(b) in \cite{HarH_book}]\label{lem:adec-modular}
Let $\phi \in \Phiw(\Omega)$. If $\phi$ satisfies \adec{q} for some $q<\infty$, then
\begin{align*}
L^{\phi}(\Omega) = \left \{f \text{ measurable } : \int_{\Omega}\phi(x, f(x)) \, dx < \infty \right \}.
\end{align*}
\end{lem}

The conjugate $\phi$-function generates the associate space of $L^{\phi}(\Omega)$ as the following Lemma shows. 

\begin{lem}[Norm conjugate formula, Theorem 3.4.6 in \cite{HarH_book}]\label{lem:norm-conjugate} If $\phi \in \Phiw(\Omega)$, then for all measurable $f$
\begin{align*}
\|f\|_{L^{\phi}(\Omega)} \approx \sup_{\|g\|_{L^{\phi^\ast}(\Omega)\leq 1}} \int_{\Omega} |f(x)g(x)| \ dx. 
\end{align*}
\end{lem}

\begin{lem}[Proposition 4.4.11 in \cite{HarH_book}]\label{lem:norm-of-ball}
Let $\phi \in \Phiw(\Omega)$ satisfy \azero{}, \aone{} and \atwo{}. Then for every ball $B \subset \Omega$ with $|B| \leq 1$ and almost every $y \in B$ we have
\begin{align*}
\|\chi_{B}\|_{L^{\phi}(\Omega)} \approx \dfrac{1}{\phi^{-1}\left (y, \frac{1}{|B|}\right )}.
\end{align*}
\end{lem}

\section{Sufficient conditions}

Let $\Omega \subset \Rn$ be open, where $n \ge 1$. For every $f \in L^1_\loc(\Omega)$ we define the (non-centered Hardy-Littlewood) maximal function by
\[
Mf(x) := \sup_{B} \frac{1}{|B|} \int_{B \cap \Omega} |f(y)| \, dy,
\]
where the supremum is taken over all open balls $B$ containing the point $x$. 
 
The maximal function $M$ is bounded from $L^\phi(\Omega)$ to $L^\phi(\Omega)$ provided that $\phi \in \Phiw(\Omega)$ satisfies \azero{}, \aone{}, \atwo{} 
and \ainc{}, see \cite[Theorem 4.3.4]{HarH_book}. This was first proved by Hästö \cite{Has15} in $\Rn$, 
see also \cite{Die05, MaeMOS13, MaeMOS13b, OhnS14}.
Let us first show that \ainc{} is not necessary for the boundedness of the maximal function. 

\begin{thm}\label{thm:bounded}
Assume that $\phi \in \Phiw(\Omega)$ satisfies \azero{}, \aone{}, \atwo{} and \ainci{}. Assume that there exists $G \subset \Omega$ such that $0\le |G| < \infty$ and $\phi$ satisfies \ainc{p} for some $p>1$ on $\Omega\setminus G$. 
Then $M:L^{\phi}(\Omega) \to L^{\phi}(\Omega)$ is bounded.  
\end{thm}

\begin{proof}Since $\phi$ satisfies \azero{}, it is equivalent with $\psi_1 \in \Phis(\Omega)$ with $\psi_1(x, 1) =1$ for almost every $x$ by \cite[Lemma 3.7.3]{HarH_book}. Thus we need to show that $M: L^{\psi_1}(\Omega) \to L^{\psi_1}(\Omega)$ is bounded.

Since $\phi\simeq \psi_1$, we obtain that $\psi_1$ satisfies \azero{}, \aone{}, \atwo{}, \ainci{} in $\Rn$ and \ainc{} in $\Omega\setminus G$. By Lemma~\ref{lem:ainc-for-smaller} we may assume that  \ainc{p} for some $p>1$ holds for all $t \ge 1$ and almost all $x$.

Let us define $\psi_2$ by
\[
\psi_2(x, t) := \begin{cases}
\psi_1(x,t)^p, &\text{ if } x \in G, t \in [0, 1];\\
\psi_1(x, t), &\text{ if } x \in G, t >1;\\
\psi_1(x, t), &\text{ if } x \in \Omega\setminus G.
\end{cases} 
\]
A short calculation show that $\psi_2 \in \Phiw(\Omega)$.
Since $\psi_2 (x, 1)=1$, it satisfies \azero{}. 

We will show that $\psi_2$ satisfies \aonep{}, which together with \azero{} yields \aone{}. 
So let us assume that  $\psi_2(y, t) \in [1, 1/|B|]$ and $x, y \in B\cap \Omega$.
Let $\beta$ be from \aonep{} of $\psi_1$. If $\beta t> 1$, then \aonep{} follows from the definition and \aonep{} of $\psi_1$. If $\beta t \le 1$, then $\psi_2(x, \beta t) \le \psi_2(x, 1) = 1 \le \psi_2 (y, t)$.

Let us then study \atwo{}.
Let $\psi_\infty$ and  $h_1\in L^1(\Omega)\cap L^\infty(\Omega)$ be from \atwo{} of $\psi_1$. Note that $\chi_G \in L^1(\Omega)\cap L^\infty(\Omega)$.
For all $t \in [0, 1]$  we have $\psi_1(x, t), \psi_2(x, t) \in [0, 1]$. We obtain for  $x \in G$ by \atwo{} of $\psi_1$ that 
\[
\psi_\infty(\beta t) \le \psi_1(x, t) + h_1(x) \le  h_1 (x) + \chi_G(x) \le \psi_2(x, t) + h_1 (x) + \chi_G(x)
\] 
and $\psi_2(x, \beta t) \le \chi_G(x) \le \psi_\infty(t) + \chi_G(x)$.
If $x \in \Omega \setminus G$, then 
\[
\psi_\infty(\beta t) \le \psi_1(x, t) + h_1(x) = \psi_2(x, t) + h_1(x)
\]
 and $\psi_2(x, \beta t)= \psi_1(x, \beta t) \le \psi_\infty(t) + h_1(x)$. Thus $\psi_2$ satisfies \atwo{} with $\psi_\infty$ 
 and $h_2 := h_1 + \chi_G$.

 Since $\psi_2(x, t) = \psi_1(x, t)$ for $t> 1$, we obtain that  $\psi_2$ satisfies \ainc{p}, $p>1$, for $t>1$ on  $\Omega$, 
and for $t>0$ on $\Omega\setminus G$.
It is also easy to see that we have \ainc{p} for $0<t \le 1$  on $G$. 
Then by continuity of $t \mapsto \psi_2(x, t)$ we obtain that $\psi_2$ satisfies \ainc{p} for $t>0$ on $\Omega$. Thus by \cite[Theorem 4.3.4]{HarH_book}
$M: L^{\psi_2}(\Omega) \to L^{\psi_2}(\Omega)$ is bounded.

Let us finally show that $L^{\psi_1}(\Omega) = L^{\psi_2}(\Omega)$ with comparable norm, since this yields that 
$M: L^{\psi_1}(\Omega) \to L^{\psi_1}(\Omega)$ is bounded. For $t\ge 1$, $\psi_1(x, t) = \psi_2(x, t)$. For $t \in[0, 1]$ we immediately obtain  that
\[
\psi_1(x, t) \le \psi_2(x, t) + \chi_G(x) \quad \text{and} \quad  \psi_2(x, t) \le \psi_1(x, t) + \chi_G(x).
\]
Thus $\psi_1$ and $\psi_2$ are weakly equivalent, and by \cite[Corollary 3.2.7]{HarH_book} $L^{\psi_1}(\Omega) = L^{\psi_2}(\Omega)$ with comparable norms.
\end{proof}

Next we give an example of a generalized $\Phi$-function that satisfies the assumptions of Theorem~\ref{thm:bounded}.

\begin{eg}\label{eg:phi-without-ainc}
Let
\begin{align*}
G:=\left \{(x,y) : x \geq 1 \text{ and } -\frac{1}{x^2}\leq y \leq \frac{1}{x^2} \right \}.
\end{align*} 
It is clear that the set $G$ is unbounded and $|G| = 2$.
Let us define $\phi: \R^2 \times[0, \infty) \to [0, \infty)$ by
\[
\phi(x, t) :=
\begin{cases}
t, &\text{ if } x \in G \text{ and }t \in [0, 1];\\
t^2, &\text{ if } x \in G \text{ and } t>1;\\
t^2 , &\text{ if } x \not\in G. 
\end{cases}
\]
A short calculation gives  $\phi \in \Phi_s(\R^2)$.
 Note that  $\phi$ is not an N-function since  $\lim_{t \to 0} \frac{\phi(x, t)}{t}=1 \neq 0$ for $x \in G$. 
We will show that
\begin{itemize}
\item[(a)] $\phi$ satisfies \azero{}, \aone{}, \atwo{} and \ainci{};
\item[(b)] $\phi$ does not satisfy \ainc{} and it fails on an unbounded set with a finite measure;
\item[(c)] $M:L^\phi (\R^2) \to L^\phi (\R^2)$ is bounded.
\end{itemize}

(a) \azero{} is clear since $\phi(x, 1) =1$ for all $x\in \R^2$. For all $\phi(x, t) \ge 1$, i.e. for all $t\ge 1$, and all $x, y \in \R^2$ we have $\phi(x, t) = \phi (y, t)$, and thus \aonep{} holds.  These together yield \aone{}.

Let us define that $\phi_\infty(t) := t^2$ and  $h:= \chi_G$. Then $h \in L^1(\R^2) \cap L^\infty (\R^2)$, and 
\[
\phi_\infty (t) \le \phi(x, t) \quad \text{and} \quad \phi(x, t) \le \phi_\infty(t) + h(x)
\]
for all $x \in \R^2$ and all $t \in [0, 1]$, and thus especially if $\phi(x, t), \phi_\infty(t) \in [0, 1]$. Thus \atwo{} holds with $\beta = 1$.
\ainci{2} is clear for $t>1$.

(b) $\phi$ does not satisfy \ainc{} for $t \in [0, 1]$ on $G$, since for all $p>1$ we have $\lim_{t\to 0^+}\frac{\phi(x, t)}{t^p}=\infty$.   

(c) Now the boundedness of the maximal function follows by Theorem~\ref{thm:bounded}.
\end{eg}

In a bounded domain \atwo{} is trivially satisfied. By choosing that the exceptional set $G$ is the whole bounded domain, we obtain by Theorem~\ref{thm:bounded} the following corollary. 

\begin{cor}\label{cor:ainci-bounded}
Let $\Omega \subset \Rn$ be bounded. Assume that $\phi \in \Phiw(\Omega)$ satisfies \azero{}, \aone{} and \ainci{}.
Then $M: L^\phi(\Omega) \to L^\phi(\Omega)$ is bounded.
\end{cor}

The next example shows that  if the exceptional set $G$  has an infinite measure, then the maximal function is not necessarily bounded.
 More precisely we give a generalized $\Phi$-function $\phi$ that satisfies \azero{}, \aone{}, \atwo{} and \ainci{}, and show that $M:L^{\phi}(\Rn) \to L^{\phi}(\Rn)$ is not bounded. Here \ainc{}  fails for every $x \in \Rn$. The key difference to previous example is that now the asymptote $\phi_\infty$ has linear growth.
 
\begin{eg}\label{eg:infinite-G}
Suppose that the dimension $n$ is at least $2$. Let us define 
\[
\phi(x, t) := 
\begin{cases}
 t, & \text{if } 0\le t \le \frac1{|x|+2}; \\
t^2 +t, & \text{if }  t >  \frac1{|x|+2}.
\end{cases}
\]
Now we have
\[
\frac{\phi(x, t)}{t} =
\begin{cases}
 1, & \text{if } 0\le t \le \frac1{|x|+2}; \\
t +1, & \text{if }  t >  \frac1{|x|+2}.
\end{cases}
\]
and hence $\phi$ satisfies \inc{1}. The other conditions of generalized $\Phi$-function are clear, and thus  $\phi \in \Phiw(\R^n)$. Moreover $\phi$ is left-continuous, but not convex.

$\phi$ satisfies \azero{} since $\phi(x, 1) = 2$ for all $x \in \Rn$.
If $\phi(y, t) \ge 1$, then $\phi(y, t) = t + t^2$, and hence $\phi(x, \beta t) \le t + t^2 = \phi(y, t)$ for all $x, y \in \Rn$. Thus \aonep{} holds, and this together \azero{} yield \aone{}.

Let $\phi_\infty (t) :=t$ and $h \equiv 0$. Then 
\[
\phi_\infty (t) \le \phi(x, t) \text{ and } \phi\Big(x, \frac{t}2\Big)  \le \frac{t^2}4 + \frac{t}2\le \frac{t}4 + \frac{t}2 \le \phi_\infty (t) 
\]
for all $t\in[0, 1]$. Therefore these inequalities are especially satisfied for $t$ such that $\phi(x,t), \phi_{\infty}(t) \in [0,1]$ for all $x$ and thus \atwo{} holds with $\beta =\frac{1}{2}$. 

Now we have verified that $\phi \in \Phiw(\Omega)$ satisfies \azero{}, \aone{} and \atwo{}. Since
\[
\frac{\phi(x, t)}{t^2} \le 1 + \frac1t \le 2 \le 2\Big( \frac{s^2+s}{s^2}\Big)= 2 \frac{\phi(x, s)}{s^2}
\]
for $1\le t<s$, we find that $\phi$ satisfies \ainci{2} with a constant $2$.
Similarly we can show that  $\phi$ satisfies \adec{2} for $t>0$. Since for every $p>1$ and every  $x \in \Rn$, we have
\[
\lim_{t \to 0^+} \frac{\phi(x, t)}{t^p} = \lim_{t \to 0^+} \frac{t}{t^p}=\infty,
\]
$\phi$ does not satisfy \ainc{}.

Let us choose $f(x) := \chi_{B(0,1)}(x)$,  and note that $f \in L^\phi(\Rn)$ due to \azero{}. Standard calculations yield $Mf(x) \approx \frac{1}{(|x|+1)^n}$. As $n>1$, we know that $|x|^n$ grows faster than $|x|$ and so we obtain a radius $R$ such that $\frac{c}{(|x|+1)^n} \le \frac1{|x|+2}$ for all $x \in \Rn\setminus B(0, R)$. Therefore, for such $x$, the maximal function is small enough to guarantee $\phi(x, Mf(x)) = \frac{c}{(|x|+1)^n}$ but large enough that
\[
\int_\Rn \phi(x,  Mf(x)) \, dx  \ge \int_{\Rn\setminus B(0, R)}\frac{c}{(|x|+1)^n}\, dx = \infty.
\] 
Since $\phi$ satisfies \adec{2} this yields that  $Mf \not \in L^\phi(\Rn)$ due to Lemma \ref{lem:adec-modular}.
\end{eg}

The next example shows that in some cases the maximal function is bounded even if the exceptional set $G$  has an infinite measure.

\begin{eg}\label{eg:M-bdd-infinite-measure}
Let us define 
\[
\phi(x, t) := 
\begin{cases}
\frac{t}{(|x|+1)^n}, & \text{if } 0\le t \le \frac1{(|x|+1)^n}; \\
t^2, & \text{if }  t >  \frac1{(|x|+1)^n}.
\end{cases}
\]
Now we have
\[
\frac{\phi(x, t)}{t} =
\begin{cases}
 \frac{1}{(|x|+1)^n}, & \text{if } 0\le t \le \frac1{(|x|+1)^n} \\
t, & \text{if }  t >  \frac1{(|x|+1)^n}.
\end{cases}
\]
and hence $\phi$ satisfies \inc{1}. The other conditions of generalized $\Phi$-function are clear, and thus  $\phi \in \Phiw(\R^n)$. Moreover $\phi$ is continuous, but not necessarily convex.

We have
\begin{equation}\label{equ:L-2}
t^2 \le \phi(x, t) \le t^2 + \frac1{(|x|+1)^{2n}}
\end{equation}
for every $t>0$ and $x \in \R^n$. Let us denote $h:= \frac1{(|x|+1)^{2n}}$, and note, by standard 
calculations, that $h \in L^1(\Rn) \cap L^\infty(\Rn)$. Hence we have $\phi \sim t^2$, and thus by 
\cite[Corollary 3.2.7]{HarH_book} $L^{\phi}(\Rn) = L^{2}(\Rn)$ with comparable norm. Hence $M:L^\phi (\Rn) \to L^\phi(\Rn)$ is bounded.

\azero{} holds since $\phi(x, 1)=1$ for all $x$. \aone{} holds since $\phi^{-1}(x, t) = t^{1/2}$ is independent of $x$ for all $t\ge 1$. \atwo{} holds by \eqref{equ:L-2} with $\phi_\infty(t) :=t^2$ and $h$.  \ainc{} fails since for all $p>1$ and all $x$ we have $\lim_{t\to 0^+} \frac{\phi(x, t)}{t^p}=\infty$. 
\end{eg}

\section{Necessary conditions}

Our first main result shows that boundedness of maximal function under assumptions \azero{}, \aone{} and \atwo{} 
implies that $\phi$ satisfies \ainci{}.  The proof is a generalization of the proof by Gallardo \cite{Gal88} and it is   
based on estimates of norms of characteristic 
functions in suitable balls and annuli by means of the conjugate $\Phi$-function. The proof fails for small
$t$ since there is no way to estimate $\phi(x,t)$ by $\phi(y,t)$ for $t \leq 1$ with \aone{}.

Let us write $\Omega_\ve := \{x \in \Omega: \dist (x, \Rn\setminus \Omega)>\ve\}$.

\begin{thm}\label{thm:ainc-infty}
Let $\phi \in \Phiw(\Omega)$ satisfy \azero{}, \aone{} and \atwo{},  and suppose that the maximal function is bounded from $L^{\phi}(\Omega)$ to $L^{\phi}(\Omega)$. Then for every $\ve >0$ we obtain that $\phi|_{\Omega_\ve}$ satisfies \ainci{}.
\end{thm}

By choosing $\Omega = \Rn$ in the theorem, we obtain that $\phi$ satisfies \ainci{}.

\begin{proof}
Assume first that $\phi \in \Phis(\Omega)$.
By assumption $\|Mf\|_{\phi} \leq K \|f\|_{\phi}$ for some $K\geq 1$ and all $f \in L^\phi (\Omega)$. 
Let  $t \ge 1$ and $s >1$, define $\omega_n := |B(0,1)|$ and let $ B^{t}(x_0):=  B\Big (x_0, \big (\frac{1}{\omega_n t }\big )^{1/n}\Big )$. A direct calculation yields
\begin{align}\label{eq:measures_of_E}
|B^{t}(x_0)| = \omega_n \left (\dfrac{1}{(\omega_n t)^{1/n}}\right )^{n} = \dfrac{1}{t}.
\end{align}
Let us choose $t_0\ge 1$ so that $3B^{t_0} (x_0) \subset \Omega$ for all $x_0 \in \Omega_\ve$. For now on assume that $t\ge t_0$.

We also have the inclusion $B^{ts}(x_0) \subset B(x, 2|x-x_0|)$ for all $x  \in B^t(x_0) \setminus B^{ts}(x_0)$. Therefore, if we choose the ball $B(x, 2|x-x_0|)$ in the maximal function, we get a pointwise estimate
\begin{align}
\begin{split}\label{eq:M-chi-pointwise}
M \chi_{B^{ts}}(x) &= \sup_{r>0} \fint_{B(x,r)} \chi_{B^{ts}}(y) \, dy \geq \dfrac{1}{|B(x,2|x-x_0|)|} \int_{B(x,2|x-x_0|)} \chi_{B^{ts}}(y) \, dy \\
&= \dfrac{1}{2^{n} \omega_n |x-x_0|^{n}} |B^{ts}| = \dfrac{1}{2^{n} \omega_n |x-x_0|^{n} ts}
\end{split}
\end{align}
for  $x  \in B^t(x_0) \setminus B^{ts}(x_0)$.

Choose $g(x) := \left (\phi^{\ast}\right )^{-1}(x, t) \chi_{B^{t}(x_0)\setminus B^{ts}(x_0)}(x)$,  and note that $\|g\|_{\phi^{\ast}} \le 1$ because
\begin{align*}
\int_{B^t(x_0) \setminus B^{ts}(x_0)} \phi^{\ast} \left (x, (\phi^\ast)^{-1}(x, t)\right ) \, dx \le \int_{B^t(x_0)}  t \, dx = |B^t(x_0)| \cdot t = 1.
\end{align*}
Note that this also yields that $\left (\phi^{\ast}\right )^{-1}(x, t)$ is finite for all $t> t_0$ and almost all $x \in B^{t}(x_0)\setminus B^{ts}(x_0)$.
Since $\|g\|_{\phi^{\ast}} \le 1$ we use $g$ as a test function for the $\phi$-norm of $M\chi_{B^{ts}}$ in the Norm conjugate formula, Lemma \ref{lem:norm-conjugate}. Thus
\begin{align*}
\|M \chi_{B^{ts}(x_0)}\|_{\phi} \gtrsim \int_{\Omega} M \chi_{B^{ts}(x_0)}(x) g(x) \, dx = \int_{B^t(x_0)\setminus B^{ts}(x_0) }  M \chi_{B^{ts}(x_0)}(x) \left (\phi^{\ast}\right )^{-1}(x, t) \, dx.
\end{align*}
Since $\phi$ satisfies \aone{}, so does $\phi^{\ast}$, as stated in Remark \ref{rem:assumptions}. By \eqref{eq:measures_of_E} we have $1 \le t =\frac1{|B^{t}(x_0)|}$ and \aone{} yields 
\begin{equation}\label{equ:A1}
\beta \left (\phi^{\ast}\right )^{-1}(y,t) \le \left (\phi^{\ast}\right )^{-1}(x,t)
\end{equation}
for all $y \in B^{t}(x_0)\setminus N$, where $|N| =0$. This additionally confirms that $\left (\phi^{\ast}\right )^{-1}(y,t)$ is finite for almost every $y \in B^t(x_0)$. We continue the estimate with \eqref{eq:M-chi-pointwise} and obtain
\begin{align*}
\|M \chi_{B^{ts}(x_0)}\|_{\phi} &\geq \beta \left (\phi^{\ast}\right )^{-1}(y,t) \int_{B^t(x_0)\setminus B^{ts}(x_0)} M \chi_{B^{ts}(x_0)}(x) \, dx\\
& \geq \beta \left (\phi^{\ast}\right )^{-1}(y,t) \dfrac{1}{2^{n} \omega_n ts} \int_{B^t(x_0)\setminus B^{ts}(x_0)} |x-x_0|^{-n}\, dx \\
& =  c(n) \beta \left (\phi^{\ast}\right )^{-1}(y,t) \dfrac{1}{2^{n} \omega_n ts} \int_{\frac1{(\omega_n s t)^{1/n}}}^{\frac1{(\omega_n t)^{1/n}}} \rho^{-n} \rho^{n-1 }\, d \rho \approx \frac{\left (\phi^{\ast}\right )^{-1}(y,t)}{ts} \ln(s).
\end{align*}

On the other hand by Lemmata \ref{lem:norm-of-ball} and \ref{lem:conjugate-inverse} we can estimate the norm of the characteristic function of a ball $B^{ts}(x_0)$:
\begin{align*}
\|\chi_{B^{ts}(x_0)}\|_{\phi} &\lesssim \dfrac{1}{\phi^{-1}\left (y, \frac{1}{|B^{ts}(x_0)|}\right )} \lesssim \dfrac{\left (\phi^{\ast}\right )^{-1}(y, \frac{1}{|B^{ts}(x_0)|})}{\frac{1}{|B^{ts}(x_0)|}}  =  |B^{ts}(x_0)| \left (\phi^{\ast}\right )^{-1}\left (y, \frac{1}{|B^{ts}(x_0)|}\right )\\ 
&= \dfrac{\left (\phi^{\ast}\right )^{-1}(y,ts)}{ts}.
\end{align*}
We combine our estimates of the norms with the boundedness assumption $\|M \chi_{B^{ts}(x_0)}\|_{L^\phi(\Omega)} \leq K \|\chi_{B_{ts}}\|_{L^\phi(\Omega)}$ and  get
\begin{align*}
\dfrac{\left (\phi^{\ast}\right )^{-1}(y,t)}{ ts} \ln (s)  \lesssim  \dfrac{\left (\phi^{\ast}\right )^{-1}(y,ts)}{ts},
\end{align*}
for all $t>t_0$ and  all $y \in B^{t}(x_0)\setminus N$. Recall that the right hand side in the inequality was deduced to be finite.

By $\phi^{-1}(x,t) \left (\phi^{\ast}\right )^{-1}(x,t) \approx t$, Lemma~\ref{lem:conjugate-inverse}, we obtain
\begin{align*}
\phi^{-1}(y,ts) \ln(s) \le Cs \phi^{-1}(y,t)
\end{align*}
for all $t>t_0$ and all $y \in B^{t}(x_0)\setminus N$. Here the constant $C$ is independent of $t$.
Choose $s:= e^{2C}>1$ so the previous inequality becomes
\begin{align*}
2\phi^{-1}(y, e^{2C} t) \leq e^{2C} \phi^{-1}(y,t).
\end{align*}
Next we choose $t'$ such that $t_0 < t :=\phi(y,t')<\infty$. 
Note that since $\phi$ is a continuous $\Phi$-function this is possible for every $t>0$, and almost every $y$, where the exceptional set  $N'$ is independent of $t$.
The lower bound is satisfied by \azero{} and \ainc{1} provided that $t' \geq at_0 /\beta$:
\[
\frac{\phi(y,t')}{t'} \ge \frac1a \frac{\phi(y,1/\beta)}{1/\beta} \ge  \frac{\beta}{a}.
\] 
The situation where $\phi(y,t') = \infty$ is considered later. With the substitution we have
\begin{align*}
2\phi^{-1}(y,e^{2C}\phi(y,t')) \leq e^{2C} \phi^{-1}(y,\phi(y,t')) \leq e^{2C} t'
\end{align*}
where the last inequality follows from Lemma~\ref{lem:composition}(a). As $\phi(y,\cdot)$ is increasing, we apply it to the both side of the inequality to get
\begin{align*}
\phi(y,\phi^{-1}(y, e^{2C}\phi(y,t'))) \leq \phi\big(y,\tfrac12e^{2C}t'\big).
\end{align*}
Since $\phi \in \Phis(\Omega)$, this yields by Lemma~\ref{lem:composition} that
\begin{align}\label{equ:Sun-Sig}
 e^{2C} \phi(y,t') \le \phi\big(y, \tfrac12e^{2C} t'\big) 
\end{align} 
for all $t'>at_0/\beta$ with $\phi(y, t')<\infty$, and  all $y \in B^t(x_0)\setminus (N\cup N')$,
where $t= \phi(y, t')$ and $|N \cup N'|=0$.

Note that in \eqref{equ:Sun-Sig} the size of the ball, variable $t$, and $t'$ are connected by the equality $t= \phi(y, t')$, and thus when $t'$ grows the size of the ball shrinks and vice versa. However two observations help us. First the equality can be replaced by the inequality $t \ge \phi(y, t')$, and secondly \eqref{equ:Sun-Sig} holds trivially if the right hand side is infinity. Then we do this with details.

The inequality \eqref{equ:Sun-Sig} trivially holds also  for  all $y \in \Omega$ with $\phi(y,t')=\infty$.  
Recall that the constant $C$ is independent of $t$. The exceptional set in \aone{} is same for all $t$, with $1\le t \le \frac1{|B|}$. Moreover if $t_1<t_2$, then $B^{t_2}(x_0) \subset B^{t_1}(x_0)$, and thus  we actually have  for a fixed $t$ that  \eqref{equ:Sun-Sig} holds for 
  all $y \in B^t(x_0)\setminus (N\cup N')$ with $t'>at_0/\beta$, and $\phi(y, t') \le t$ or  $\phi(y, t') = \infty$, where $N$ is from \eqref{equ:A1} and $N'$ is from \eqref{equ:Sun-Sig}.  
 Since $\phi$ is continuous $t'$ can have all the values from the range $(at_0/\beta, \sup\{s: \phi(y, s)=t\})$.
 
For a fixed $t$ we cover $\Omega_\ve$   by finite number of balls $B^t(x_i) \Subset 3B^{t_0}(x_i) \subset \Omega$, $x_i \in \Omega_\ve$ and $i=1,2, 3, \ldots, k$, and obtain that \eqref{equ:Sun-Sig} holds for 
 all $y \in \Omega_\ve\setminus N_t$ with $t'>at_0/\beta$, and $\phi(y, t') \le t$, where $|N_t| =0$, and for all $y\in\Omega_\ve$ with  $\phi(y, t') = \infty$. 
Finally by taking a sequence $(t_i)$ with $t_i \to \infty$, we see that \eqref{equ:Sun-Sig} holds for all $t'>at_0/\beta$ and  all $y$ in 
\[
\bigcup_{i=1}^\infty \Big(\{y \in \Omega_\ve: \phi(y, t') \le t_i \}\setminus N_{t_i} \Big) \cup \{y \in \Omega_\ve: \phi(y, t')=\infty \} = \Omega_\ve \setminus \bigcup_{i=1}^\infty N_{t_i}.
\]
 Thus we have that \eqref{equ:Sun-Sig} holds for every  $t'>at_0/\beta$ and almost every $x$ and $y$.
 Now Lemma~\ref{lem:Sun-Sig} yields that $\phi|_{\Omega_\ve}$ satisfies \ainci{}. 
 
 Assume then that $\phi\in \Phiw(\Omega)$. Then by \cite[Theorem 2.2.3]{HarH_book} there exists $\psi \in \Phis(\Omega)$ such that $\phi \simeq \psi$. By the first part of the proof we see that $\psi$ satisfies \ainci{}, and thus by Lemma~\ref{lem:ainc-equivalent} $\phi$  satisfies \ainci{}. 
\end{proof}

Next result generalize Gallardo's result \cite{Gal88} from $N$-functions to all weak $\Phi$-functions.

\begin{cor}\label{cor:Orlicz-case}
Let $\phi \in \Phiw$ and $M: L^\phi(\Rn) \to L^\phi(\Rn)$ be bounded. Then $\phi$ satisfies \ainc{}.
\end{cor}

\begin{proof}
Since $\phi$ is independent of $x$ it satisfies \azero{}, \aone{} and \atwo{} so Theorem~\ref{thm:ainc-infty} yields \ainci{} for $\phi$. In the proof \aone{} is used to remove   $(\phi^{\ast} )^{-1}(x, t)$ out from the integral. Since $\phi$ is independent of $x$, this can be done for  all $t$. Also, as we assume that the maximal function is bounded in $\Rn$, we do not need to assume that $t> t_0$ like was done for  domain $\Omega$. Hence the same proof shows that $\phi$ satisfies \ainc{}.
\end{proof}

In Corollary~\ref{cor:ainci-bounded} we show that in a bounded domain \azero{}, \aone{} and \ainci{} implies 
the boundedness of the maximal function. Next example shows that \ainci{} can not be replaced by the condition of Theorem~\ref{thm:ainc-infty}, i.e. $\phi|_{\Omega_\ve}$ satisfies \ainci{}.

\begin{eg}
Let $\Omega := (-2, 2) \subset \R$. Let us define $p: \overline \Omega \to [1, 2]$ by $p(-2) = 1=p(2)$, $p(x) =2$ for $x \in [-1, 1]$ and $p$ is linear in $[-2, -1]$ and in $[1, 2]$. Then $p$ is Lipschitz-continuous. Let $\phi:\Omega \times [0, \infty) \to [0, \infty)$, $\phi(x, t) := t^{p(x)}$.

Since $1^{p(x)}=1$ for all $x$, $\phi$ satisfies \azero{}. Since $p$ is Lipschitz-continuous, it is $\log$-Hölder continuous and thus \aone{} holds for $\phi$,
\cite[Proposition 7.1.2]{HarH_book}. Let $\ve \in(0, 1)$. In $\Omega_\ve = (-2+\ve, 2-\ve)$ we have $p(x) >  1+ \ve$, and thus $\phi$ satisfies \inc{1+\ve}.
However $p^-= \inf p =1$ and thus by \cite[Theorem 4.7.1 and Remark 4.7.2]{DieHHR11} the maximal operator is not bounded from $L^\phi(\Omega)$ to itself.
\end{eg}

The next results shows that $\phi_\infty$ from \atwo{} satisfies \ainc{} in $[0,s]$ for small $s$.

\begin{prop}\label{prop:p_infty}
Let $\phi \in \Phiw(\Rn)$ satisfy \azero{} and \atwo{},  and suppose that the maximal function is bounded from $L^{\phi}(\Rn)$ to $L^{\phi}(\Rn)$. Then every $\phi_\infty$ satisfies \ainc{} for $t\in[0, s]$ for some $s>0$. Additionally, every $\phi_\infty$ can be modified to satisfy \ainc{} while preserving the \atwo{} property.
\end{prop}

Note that by Lemma~\ref{lem:ainc-for-smaller} we may increase $s$ to be any finite number.

\begin{proof}
Since $\phi$ satisfies \azero{}, we obtain by \cite[Lemma 4.2.9]{HarH_book} that \atwo{} is equivalent with
\begin{align}\label{eq:L-phi-infty}
L^{\phi}(\Rn) \cap L^{\infty}(\Rn) = L^{\phi_\infty}(\Rn) \cap L^{\infty}(\Rn).
\end{align}
Note that the assumption ''$\phi_\infty$ satisfies \azero'' in \cite[Lemma 4.2.9]{HarH_book} is trivially satisfied, since $\phi_\infty \in \Phiw$, and that the norms of these spaces are comparable. 

Since $\phi$ and $\phi_\infty$ satisfy \azero{}, there exists $\beta_0>0$ such that 
$\phi(x,\beta_0)\le 1 \le \phi(x,\frac1{\beta_0})$ and 
 $\phi_\infty(\beta_0)\le 1 \le \phi_\infty(\frac1{\beta_0})$.
We define $\psi\in\Phiw$ by 
\[
\psi(t):= \max\{\phi_\infty(t), \infty \chi_{(\beta_0,\infty)}(t)\}.
\]
Then $L^{\phi_\infty}(\Rn)\cap L^\infty(\Rn) = L^\psi(\Rn)$ with comparable norms. Together with \eqref{eq:L-phi-infty} we have that  $L^{\phi}(\Rn)\cap L^\infty(\Rn)$ and $L^{\psi}(\Rn)$ have comparable norms. Since the maximal function is trivially bounded from  $L^{\infty}(\Rn) \to L^{\infty}(\Rn)$ and bounded from $L^{\phi}(\Rn) \to L^{\phi}(\Rn)$ by assumption, we obtain that $M: L^\psi(\Rn) \to L^\psi(\Rn)$ is bounded.

By Corollary~\ref{cor:Orlicz-case} we obtain that $\psi$ satisfies \ainc{}. But this yields that $\phi_\infty$ satisfies \ainc{} for $[0, s]$ where $s = \beta_0$, as was to be proved.

Next we show that $\phi_\infty$ can be modified so that it satisfies \ainc{} but remains valid for \atwo{}. 
By assumption $\phi$ satisfies \atwo{} i.e  the condition is satisfied for all $t$ such that $\phi_\infty(t) \in [0,1]$ and $\phi(x,t) \in [0,1]$ for almost every $x \in \Omega$. Let 
\begin{align*}
t_1 := \frac12 \sup \{t: \max\{\phi_\infty(t), \phi(x,t)\} \leq 1 \text{ for a.e } x \in \Omega\}.
\end{align*}
By \azero{} we have   $\phi(x, \beta)\le 1$, and hence $t_1 >0$.
Now $\phi(x,t), \phi_\infty(t) \in [0,1]$ for all $t \in [0,t_1]$. Thus  $\phi$ satisfies \atwo{} with $\phi_\infty$ for all $t \in [0,t_1]$.

 We have proved that all asymptotes $\phi_\infty$ satisfy \ainc{p}, where $p>1$, in $[0,s]$ and by Lemma \ref{lem:ainc-for-smaller} we may assume that $s=t_1$. We define a new asymptote
\begin{align*}
\tilde \phi_\infty (t):=
\begin{cases}
\phi_\infty(t), &t \in [0,t_1]; \\
\phi_\infty(t_1) + (t-t_1)^p, & t >t_1.
\end{cases}
\end{align*}
A short calculation shows that $\tilde{\phi}_\infty \in \Phiw$ and it satisfies \ainc{p}. As $\tilde{\phi}_\infty(t) = \phi_\infty(t)$ when $t \in [0,t_1]$, it is a valid asymptote of $\phi$ for \atwo{}. Thus we can always choose an asymptote $\tilde{\phi}_\infty$ that  satisfies \ainc{}.
\end{proof}

Every asymptote $\phi_\infty$ might not necessarily satisfy \ainci{} but we can always modify it for large $t$ so that it satisfies \ainc{}. 
However, if $h$ from \atwo{} satisfies $\lim_{|x|\to \infty}h(x) =0$, we can choose a more natural asymptote that satisfies \ainc{}.

\begin{rem}\label{rem:h=0}
Let $\phi \in \Phiw(\Rn)$ satisfy \azero{}, \aone{} and \atwo{} and suppose that the maximal function is bounded from $L^{\phi}(\Rn)$ to $L^{\phi}(\Rn)$.
Let $h\in L^1(\Rn) \cap L^\infty (\Rn)$ be from \atwo{} of  $\phi$, and assume additionally that $\lim_{|x| \to \infty} h(x) =0$. In this case we can choose that
$\phi_\infty = \phi^+_\infty$ or $\phi_\infty = \phi^-_\infty$, where
\begin{align*}
\phi^+_\infty (t) := \limsup_{|x| \to \infty} \phi(x, t) \quad \text{and} \quad \phi^-_\infty (t) := \liminf_{|x| \to \infty} \phi(x, t). 
\end{align*}
$\phi_{\infty}^{+}, \phi_{\infty}^{-} \in \Phiw$ since $\phi$ satisfies \azero{} \cite[Lemma 2.5.18]{HarH_book}.

Let us choose that $\phi_\infty = \phi^+_\infty$. By Proposition~\ref{prop:p_infty} $\phi^+_\infty$ satisfies \ainc{p_1} for $t \in [0, s]$ for some $p_1 >1$ and $s>0$ and by Theorem~\ref{thm:ainc-infty} $\phi$ satisfies \ainci{p_2} for $t \in [t_0,\infty)$ and for some $p_2 >1$. Let $p := \min \{p_1, p_2\}$ and notice that $\phi$ satisfies \ainc{p} for $t \in [t_0, \infty)$ and $\phi_\infty^{+}$ satisfies \ainc{p} for $t \in [0, s]$. By Lemma \ref{lem:ainc-for-smaller} we can decrease $t_0$ to be smaller than $s$.  Then we obtain for
$t_0 \le t'<s'$ that
\[
\limsup_{|x|\to \infty} \frac{\phi(x, t')}{t'^p} \le a \limsup_{|x|\to \infty} \frac{\phi(x, s')}{s'^p},
\] 
 and hence $\phi^+_\infty$ satisfies \ainc{p} for $t \in [t_0, \infty)$ as well. Since it already satisfied \ainc{p} for $[0,s]$, it satisfies \ainc{p} for all $t\geq 0$ for some $p>1$.
 Similarly we can show that $\phi^-_\infty$ satisfies \ainc{}. 
 \end{rem}

Colombo and Mingione proved in the double phase case that the maximal function is bounded provided that 
$1 < p \leq q <\infty$,  $a \in C^{0,\alpha}(\Omega)$ is bounded and $\frac{q}{p} \leq 1 + \frac{\alpha}{n}$ , \cite{ColM15a}. 
The next corollary  shows that $p>1$ is necessary.

\begin{cor}\label{cor:dp-case}
Let $\phi(x,t) := t^{p} + a(x) t^{q}$. 
If the maximal function is bounded from $L^{\phi}(\Omega)$ to $L^{\phi}(\Omega)$, then $p>1$.
\end{cor}

\begin{proof}
We obtain by \cite[Propositions 7.2.1]{HarH_book} that $\phi$ satisfies \azero{} and \atwo{} with  $\phi_{\infty}(t) = t^p$ and $h \equiv 0$.
Therefore Proposition~\ref{prop:p_infty} shows that $\phi_{\infty}(t) = t^p$ satisfies \ainc{} for small $t$ i.e. $p >1$.
\end{proof}

Our second main result gives a sharp growth condition for boundedness of the maximal function.

\begin{thm}\label{thm:weak}
Let $\phi \in \Phiw(\Rn)$ satisfy \azero{}, \aone{} and \atwo{}. Then the Hardy--Littlewood maximal function is bounded from $L^{\phi}(\Rn)$ to $L^{\phi}(\Rn)$ if and only if  there exists $\psi \in \Phiw(\Rn)$ such that 
$\phi$ and $\psi$ are weakly equivalent and $\psi$ satisfies \azero{}, \aone{}, \atwo{} and \ainc{}.
\end{thm}

\begin{proof}
Assume first that $\psi$ satisfies \azero{}, \aone{}, \atwo{} and \ainc{}. Then by  \cite[Theorem 4.3.4]{HarH_book} $M: L^\psi(\Rn) \to L^\psi(\Rn)$ is bounded. Since $\psi \sim \phi$, we have $L^\psi(\Rn) = L^\phi(\Rn)$ with comparable norms, and thus $M: L^\phi(\Rn) \to L^\phi(\Rn)$ is bounded.

Assume then that $\phi \in \Phiw(\Rn)$ satisfies \azero{} with constant $\beta_0$, \aone{} and \atwo{}, and $M: L^\phi(\Rn) \to L^\phi(\Rn)$ is bounded.
Let  $\phi_\infty \in \Phiw$, $h\in L^1(\Rn) \cap L^{\infty}(\Rn)$, $\beta_2 <1$ and $s>0$ be from \atwo{} of $\phi$. Define
\begin{align*}
t_1 := \frac12 \sup \{t: \max\{\phi_\infty(t), \phi(x,t)\} \leq 1 \text{ for a.e } x \in \Omega\}.
\end{align*}
Since $\phi$ satisfies \azero{}, we have $\phi(x, \beta_0) \le 1$ for almost all $x$, and  thus $t_1 >0$.
Moreover $\phi_\infty(t_1)\le 1$ and $\phi(x, t_1) \le 1$ for almost every $x$. 
As in the proof of Proposition \ref{prop:p_infty}, we see that \atwo{} of $\phi$ holds for some $s'>0$ when we restrict ourselves to $t \in [0,t_1]$. 
We define
\begin{align*}
\psi(x,t) := 
\begin{cases}
\phi_\infty(\beta_2 t), &t \in [0, t_1] \\
\phi_\infty(\beta_2 t_1) + \max\{\phi(x,t)-\phi_\infty(\beta_2 t_1),0\}, & t >t_1.
\end{cases}
\end{align*}
We need to prove that it satisfies the properties in the claim.

Let us first show that $\psi$ satisfies \ainc{}. By Proposition~\ref{prop:p_infty} and Lemma~\ref{lem:ainc-for-smaller} $\phi_\infty$ satisfies \ainc{} 
for $t \in (0, t_1]$. Since $\phi_\infty(\beta_2 t_1) \le \phi_\infty(\beta_2 t_1) + \max\{\phi(x,t)-\phi_\infty(\beta_2 t_1),0\}$, we need only show that
$t \mapsto \phi_\infty(\beta_2 t_1) + \max\{\phi(x,t)-\phi_\infty(\beta_2 t_1),0\}$ satisfies \ainc{} for $t \in [t_1, \infty)$. 
Thus by taking smaller of the two exponent  we have \ainc{} for $\psi$.

Let us then study $t \mapsto \phi_\infty(\beta_2 t_1) + \max\{\phi(x,t)-\phi_\infty(\beta_2 t_1),0\}$ on $[t_1, \infty)$.
By Theorem~\ref{thm:ainc-infty} and Lemma~\ref{lem:ainc-for-smaller} $\phi$ satisfies \ainc{p}, $p>1$ for 
$t \in [t_1, \infty)$ with a constant $a$.  
\azero{} yields  $\phi(x, 1/\beta_0) \ge 1$ for almost every $x$. Thus for $t \ge 1/\beta_0$, we have
$\phi_\infty(\beta_2 t_1) + \max\{\phi(x,t)-\phi_\infty(\beta_2 t_1),0\} = \phi(x, t)$, and hence \ainc{p} is clear for $t\ge 1/\beta_0$. 
If $t_1\le s< t\le 1/\beta_0$, then
\begin{equation*}
\begin{split}
 &\frac{\phi_\infty(\beta_2 t_1) + \max\{\phi(x,s)-\phi_\infty(\beta_2 t_1),0\}}{s^p}\\
&\qquad\le  \frac{t^p}{s^p} \frac{\phi_\infty(\beta_2 t_1)}{t^p} + \max\Big\{a \frac{\phi(x,t)}{t^p}- \frac{t^p}{s^p} \frac{\phi_\infty(\beta_2 t_1)}{t^p},0\Big\}.
\end{split}
\end{equation*}
Depending which one is greater in the maximum, we have two cases:
\[
\begin{split}
 \frac{\phi_\infty(\beta_2 t_1) + \max\{\phi(x,s)-\phi_\infty(\beta_2 t_1),0\}}{s^p} &\le \frac{t^p}{s^p} \frac{\phi_\infty(\beta_2 t_1)}{t^p} \le  \frac{1}{(\beta_0 t_1)^p} \frac{\phi_\infty(\beta_2 t_1)}{t^p}\\
&\le   \frac{1}{(\beta_0 t_1)^p}  \frac{\phi_\infty(\beta_2 t_1) + \max\{\phi(x,t)-\phi_\infty(\beta_2 t_1),0\}}{t^p}
\end{split}
\]
or 
\[
\begin{split}
 \frac{\phi_\infty(\beta_2 t_1) + \max\{\phi(x,s)-\phi_\infty(\beta_2 t_1),0\}}{s^p} &\le a \frac{\phi(x,t)}{t^p}\\ 
 &\le  a \frac{\phi_\infty(\beta_2 t_1) + \max\{\phi(x,t)-\phi_\infty(\beta_2 t_1),0\}}{t^p}.
\end{split}
\]
This completes the case $t \in [t_1, 1/\beta_0]$.

Other properties of weak $\Phi$-function follows easily, since both $\phi$ and $\phi_\infty$ are weak $\Phi$-functions, and thus $\psi \in \Phiw(\Omega)$. 

 Clearly $\psi(x,t_1)\leq 1$. Let  $\beta_0$ be from \azero{} of $\phi$, as before.  Then $\phi(x, 1/\beta_0) \ge 1$ and hence $t_1 < 1/\beta_0$. Thus we obtain  $\psi(x, 1/\beta_0) = \phi(x, 1/\beta_0) \ge 1$. 
These yield that $\psi$ satisfies \azero{}.
Now because firstly $\phi$ satisfies \aonep{} and $\psi(x,t) = \phi(x,t)$, when $\psi(x,t)\geq 1$, 
and secondly $\psi$ satisfies \azero{}, we conclude that $\psi$ satisfies \aone{}. 
As $\psi(x,t) = \phi_\infty(\beta_2 t)$ and $\phi_\infty(\beta_2 t) \in [0,s']$ when $t \in [0,t_1]$, 
we see that $\psi$ satisfies \atwo{} with $\phi_\infty(\beta_2 t)$, $h\equiv 0$ and a constant $1$.

Let us finally prove that $\phi \sim \psi$. If $t \in [0,t_1]$, then by \atwo{} of $\phi$ we get 
\begin{align*}
\phi(x,\beta_2^2 t) \leq \phi_\infty(\beta_2 t) + h(x) =  \psi(x,t) + h(x).
\end{align*}
If on the other hand $t > t_1$, then
\begin{align*}
\phi(x,\beta_2^2 t) \leq \phi(x,t) \leq \psi(x,t)
\end{align*}
and therefore $\phi(x,\beta_2^2 t) \leq \psi(x,t) + h(x)$ for all $t \geq 0$. 

Let us now look the other direction. If $t \in [0, t_1]$, then by \atwo{} of $\phi$ we get 
\begin{align*}
\psi(x,t) = \phi_\infty(\beta_2 t) \leq \phi(x,t) + h(x).
\end{align*}
For $t \geq t_1$ we have two cases: either $\phi(x, t) \geq \phi_\infty( \beta_2 t_1)$ or not. If this inequality holds, then
\begin{align*}
 \psi(x, t)  = \phi_\infty(\beta_2 t_1) + \max\{\phi(x, t) - \phi_\infty(\beta_2 t_1),0\} =  \phi(x,t).
\end{align*}
Let us then consider the case $\phi(x,t) < \phi_\infty(\beta_2 t_1)$. Now since $\phi$ satisfies the condition of \atwo{} for $t = t_1$, we have
\begin{align*}
 \psi(x,t) &= \phi_\infty(\beta_2 t_1) + \max\{\phi(x,t) - \phi_\infty(\beta_2 t_1),0\} = \phi_\infty(\beta_2 t_1) \leq \phi(x,t_1) + h(x) \\
&\leq \phi(x,t) + h(x).
\end{align*}
Therefore we also have that $\psi(x, t) \leq \phi(x,t) + h(x)$ for all $t \geq 0$.
All in all, we have shown that $\phi \sim \psi$ with a constant $\beta_2^2$ and function $h$ from \atwo{} of $\phi$.
\end{proof}

\section*{Acknowledgements}
We would like to thank the anonymous referees for the helpful suggestions in improving some of the proofs. 


\bigskip

\noindent\small{\textsc{P. Harjulehto}}\\
\small{Department of Mathematics and Statistics,
FI-20014 University of Turku, Finland}\\
\footnotesize{\texttt{petteri.harjulehto@utu.fi}}\\

\noindent\small{\textsc{A. Karppinen}}\\
\small{Department of Mathematics and Statistics,
FI-20014 University of Turku, Finland}\\
\footnotesize{\texttt{arttu.a.karppinen@utu.fi}}\\

\end{document}